\documentclass[12pt]{amsart}
\usepackage{amssymb}
\usepackage[margin=1in]{geometry}

\title{$L^p$-Fourier and Fourier-Stieltjes algebas for locally compact groups}
\author{Matthew Wiersma}
\address{Department of Pure Mathematics, University of Waterloo, Waterloo, ON, Canada N2L 3G1}
\email{mwiersma@uwaterloo.ca}

\newtheorem{theorem}{Theorem}[section]
\newtheorem{corollary}[theorem]{Corollary}
\newtheorem{prop}[theorem]{Proposition}
\newtheorem{lemma}[theorem]{Lemma}

\theoremstyle{remark}
\newtheorem{remark}[theorem]{Remark}

\theoremstyle{definition}

\newtheorem{example}[theorem]{Example}

\newcommand{\fn}{\!:}
\newcommand{\C}{\mathbb C}
\newcommand{\R}{{\mathbb R}}
\newcommand{\N}{{\mathbb N}}

\newcommand{\Hi}{\mathcal{H}}
\newcommand{\lla}{\left\langle}
\newcommand{\rra}{\right\rangle}
\newcommand{\mc}{\mathcal}

\newcommand{\tn}{\textnormal}

\newcommand{\Z}{\mathbb Z}

\newcommand{\T}{\mathbb T}

\newcommand{\F}{\mathbb F}

\begin{document}

\begin{abstract}
Let $G$ be a locally compact group and $1\leq p<\infty$. A continuous unitary representation $\pi\fn G\to B(\Hi)$ of $G$ is an $L^p$-representation if the matrix coefficient functions $s\mapsto \lla \pi(s)x,x\rra$ lie in $L^p(G)$ for sufficiently many $x\in \Hi$. Brannan and Ruan \cite{br} defined the $L^p$-Fourier algebra $A_{L^p}(G)$ to be the set of matrix coefficient functions of $L^p$-representations. Similarly, the $L^p$-Fourier-Stieltjes algebra $B_{L^p}(G)$ is defined to be the weak*-closure of $A_{L^p}(G)$ in the Fourier-Stieltjes algebra $B(G)$. These are always ideals in the Fourier-Stieltjes algebra containing the Fourier algebra. In this paper we investigate how these spaces reflect properties of the underlying group and study the structural properties of these algebras. As an application of this theory, we characterize the Fourier-Stieltjes ideals of $SL(2,\R)$.
\end{abstract}

\maketitle

%\tableofcontents

\section{Introduction}

The theory of Banach algebras is motivated by examples, and many of the important examples in the field of Banach algebras arise from locally compact groups. The most classicly studied Banach algebra associated to a locally compact group $G$ is the group algebra $L^1(G)$ with multiplication given by convolution. In 1952 James Wendel showed that $L^1(G)$ is a complete invariant for locally compact groups $G$ in the sense that $L^1(G_1)$ is isometrically isomorphic to $L^1(G_2)$ as Banach algebras if and only if $G_1$ is homeomorphically isomorphic to $G_2$ \cite{wend}. Hence, we can expect that many properties of the group may be reflected in the group algebra. This is indeed the case. For example, it is easily checked that $G$ is abelian if and only if $L^1(G)$ is commutative and $G$ is a discrete group if and only if $L^1(G)$ is unital. A much less obvious property shown by Barry Johnson is that $G$ is amenable if and only if $L^1(G)$ is amenable as a Banach algebra \cite{j}.

Since the group algebra $L^1(G)$ of a locally compact group $G$ is an involutive Banach algebra, it is natural to consider operator algebras containing a copy of $L^1(G)$ as a dense subspace. The most heavily studied of these are the full and reduced group $C^*$-algebras $C^*(G)$ and $C^*_r(G)$, and the group von Neumann algebra $VN(G)$ which contain norm and weak*-dense copies of $L^1(G)$, respectively. Unlike the group algebra $L^1(G)$, these operator algebras fail to completely determine the group $G$ but are still able to encode many useful properties of the underlying group $G$. %For example, the group $G$ is amenable if and only if $C^*(G)=C^*_r(G)$ REF.

Related to the group von Neumann algebra and the full group $C^*$-algebra, we have the Fourier algebra $A(G)$ and the Fourier-Stieltjes algebra $B(G)$ which naturally identify with the predual of $VN(G)$ and the dual of $C^*(G)$, respectively. The Fourier and Fourier-Stieltjes algebras are can be viewed as subalgebras of $C_0(G)$ and $C_b(G)$, respectively, endowed with a norm dominating the uniform norm. Despite always being commutative Banach algebras even when $G$ is nonabelian, Martin Walter demonstrated that these Banach algebras $A(G)$ and $B(G)$ are complete invariants for $G$ \cite{walt}. In many ways $A(G)$ is analagous to the group algebra $L^1(G)$, however it is not the case that $A(G)$ is amenable if and only if $G$ is amenable. In fact Brian Forrest and Volker Runde showed that $A(G)$ is amenable if and only if $G$ is almost abelian, i.e., if and only if $G$ contains an open abelian subgroup of finite index \cite{fr}. Recall that $A(G)$ is the predual of $VN(G)$ and, hence, has a canonical operator space structure. By taking this observation into account, Zhong-Jin Ruan demonstrated that $G$ is amenable if and only if $A(G)$ is operator amenable \cite{ru}. The amenability of $G$ has also been characterized by the existance of a bounded approximate identity in $A(G)$ by Leptin \cite{le} and in terms of the multipliers of $A(G)$ by Losert \cite{lo}.

%The Fourier algebra encodes the amenability in a number of different ways. For example, $A(G)$ is amenable if and only if $A(G)$ has a bounded approximate identity (Leptin's characterization) if and only if the multipliers of $A(G)$ are exactly $B(G)$ (Losert's characterization) REF. It is not the case, however, that $A(G)$ is amenable as a Banach algebra if and only if $G$ is amenable. In fact, Brian Forrest and Volker Runde showed that the Fourier algebra $A(G)$ is an amenable Banach algebra if and only if $G$ is almost abelian, i.e., $G$ contains an open abelian subgroup of finite index REF. Notice that since the Fourier algebra $A(G)$ is the predual of the von Neumann algebra $VN(G)$, it has a natural operator space structure. Zhong-Jin Ruan was able to use this observation to show that a locally compact group $G$ is amenable if and only if $A(G)$ is operator amenable. The $L^p$-Fourier and Fourier-Stieltjes spaces that we will study in this paper are ideals of the Fourier Stieltjes algebra containing the Fourier algebra.

Recently Nate Brown and Erik Guentner defined the concept of $L^p$-representations and their associated $C^*$-algebras \cite{bg}. Let $G$ be a locally compact group and $1\leq p<\infty$. A (continuous unitary) representation $\pi\fn G\to B(\Hi)$ is said to be an $L^p$-representation if, roughly speaking, the matrix coefficient functions $s\mapsto \lla \pi(s)x,x\rra$ are in $L^p(G)$ for sufficiently many $x\in \Hi$. As examples, the left regular representation $\lambda$ is an $L^p$-representation of $G$ for each $1\leq p<\infty$ and the trivial representation is an $L^p$-representation if and only if $G$ is compact. When $G$ is the group $SL(2,\R)$, it is an immediate consequence of the work of Ray Kunze and Elias Stein \cite{ks} that each nontrivial irreducible represenation of $G$ is an $L^p$-represenation of $G$ for some $p\in [2,\infty)$. The $C^*$-algebra $C^*_{L^p}(G)$ is defined to be the completion of $L^1(G)$ with respect to a $C^*$-norm arising from $L^p$-represenations of the group $G$. When $p\in [1,2]$, $C^*_{L^p}(G)$ is simply the reduced group $C^*$-algebra $C^*_r(G)$, but this need not be the case for $p> 2$. Indeed, Rui Okayasu showed that the $C^*$-algebras $C^*_{\ell^p}(\F_d)$ are distinct for every $2\leq p<\infty$ \cite{o} where $\F_d$ denotes the free group on $2\leq d<\infty$ generators. In section \ref{SL(2,R)} we observe that the analogous result holds for $SL(2,\R)$.

In \cite{br} Michael Brannan and Zhong-Jin Ruan defined and developed some basic theory of $L^p$-Fourier and Fourier-Stieltjes algebras, denoted $A_{L^p}(G)$ and $B_{L^p}(G)$. Evidencing their usefulness, the $L^p$-Fourier-Stieltjes algebras were used in \cite{w2} to find many intermediate $C^*$-norms on tensor products of group $C^*$-algebras.
%The $L^p$-Fourier-Stieltjes algebras identify naturally with the dual spaces of the $C^*$-algebras $C^*_{L^p}(G)$ and the $L^p$-Fourier algebras are naturally preduals of certain von Neumann algebras.
The $L^p$-Fourier and Fourier-Stieltjes algebras are ideals of the Fourier-Stieltjes algebra corresponding to coefficient functions of $L^p$-representations. Similar to the case of the $C^*$-algebras, the $L^p$-Fourier algebra coincides with the Fourier algebra $A(G)$ and the $L^p$-Fourier-Stieltjes algebra with the reduced Fourier-Stieltjes algebra when $p\in [1,2]$. This is not the case necessarily for $p>2$. In fact we demonstrate rich classes of groups $G$ so that $A_{L^p}(G)$ is distinct for every $p\in [2,\infty)$ and $B_{L^p}(G)$ is distinct for each $p\in [2,\infty)$.  As an application of the theory developed throughout this paper, we characterize the Fourier-Stieltjes ideals of $SL(2,\R)$ in terms of $L^p$-Fourier-Stieltjes algebras.

Similar to the Fourier algebra, we show the $L^p$-Fourier algebra is a complete invariant for locally compact groups. Unlike the Fourier algebra, the $L^p$-Fourier algebra can lack many nice properties even when $G$ is a very nice group. For example, when $G$ is a noncompact abelian group, then $A_{L^p}(G)$ is not even square dense for each $p\in (2,\infty)$ and, hence, lacks any reasonable notion of amenability. So the analogues of Ruan's and Leptin's characterizations of amenability fail for the $L^p$-Fourier algebras. Though the analogues of these characterizations of amenability fail for the $L^p$-Fourier algebra, we show that the analogue of Losert's characterization of amenability in terms of multipliers \cite{lo} holds for $A_{L^p}(G)$ and Runde-Spronk's characterization of amenability in terms of operator Connes amenability \cite{rs} holds for $B_{L^p}(G)$.

This paper is organized as follows. The following two sections are dedicated to reviewing the background required for this paper. A brief overview of Fourier and Fourier-Stieltjes spaces is provided in section 2, and the theory of $L^p$-represenations is recalled in section 3. In section 4 we recall the definition of the $L^p$-Fourier and Fourier-Stieltjes algebras and note some basic properties. Section 5 is devoted to studying the $L^p$-Fourier algebras for abelian groups. In section 6 we study the structural properties of the $L^p$-Fourier and Fourier-Stieltjes algebras. As a final application of this theory, we characterize the Fourier-Stieltjes ideals of $SL(2,\R)$ in section 7.

\section{Preliminaries on Fourier and Fourier-Stieltjes spaces}\label{spaces}

In this section we provide an overview of the theory of Fourier and Fourier-Stieltjes spaces as developed by Pierre Eymard \cite{e} and Gilles Arsac \cite{a}. The reader is encouraged to see these two papers for additional details.

Let $G$ be a locally compact group. The {\it Fourier-Stieltjes algebra} is defined to be the set of coefficient functions $s\mapsto\pi_{x,y}:=\lla \pi(s)x,y\rra$ as $\pi\fn G\to B(\Hi_\pi)$ ranges over the (continuous unitary) representations of $G$ and $x,y$ over $\Hi_\pi$. Then $B(G)$ identifies naturally with the dual of the full group $C^*$-algebra $C^*(G)$ via the identification $<u,f>\,=\int_G u(s)f(s)\,ds$ for $f\in L^1(G)$ and $u\in B(G)$. When endowed with the norm attained from this identification with $C^*(G)^*$, the Fourier-Stieltjes algebra becomes a Banach algebra under pointwise operations. When $G$ is abelian, $B(G)$ is isometrically isomorphic as a Banach algebra to the measure algebra $M(\widehat{G})$ and the isomorphism is given by the Fourier-Stieltjes transform.

Let $\mc S$ be a collection of representations of $G$. The {\it Fourier space} $A_{\mc S}$ is defined to be the closed linear span of coefficient functions $\pi_{x,y}$ in $B(G)$ as $\pi$ ranges over representations in $\mc S$ and $x,y$ over $\Hi_\pi$. If $\mc S$ consists of a single representations $\pi$, then $A_\mc S$ is denoted by $A_\pi$. These Fourier spaces $A_{\mc S}$ are translation invariant (under both left and right translation) subspaces of $B(G)$ and, conversely, every closed translation invariant subspace of $B(G)$ is realizable as a Fourier space $A_\pi$ for some representation $\pi$ of $G$. A fortiori, for every collection $\mc S$ of representations of $G$, $A_{\mc S}$ is equal to $A_\pi$ for some representation $\pi$ of $G$.

As a distinguished Fourier space, the Fourier algebra $A(G)$ is defined to be $A_\lambda$ where $\lambda$ denotes the left regular representation of $G$. Although not obvious, it is a consequence of Fell's absorption principle that $A(G)$ is a subalgebra (and in fact an ideal) of $B(G)$. When $G$ is abelian $A(G)$ is isometrically isomorphic to the group algebra $L^1(\widehat{G})$ via the Fourier transform.

Fix some representation $\pi$ of $G$. The Fourier space $A_\pi$ is exactly the set of infinite sums $\sum_{n=1}^\infty \pi_{x_n,y_n}$ with $\{x_n\},\{y_n\}\subset \Hi_\pi$ satisfying the condition that $\sum_{n=1}^\infty \|x_n\|\|y\|_n<\infty$. Moreover, the norm of an element $u\in A_\pi$ is given by
$$ \|u\|=\inf\bigg\{\sum_{n=1}^\infty\|x_n\|\|y_n\| : u=\sum_{n=1}^\infty \pi_{x_n,y_n}\bigg\} $$
and this infimum is attained.

For each represntation $\pi$ of $G$, define $VN_\pi$ to be the von Neuman algebra $\pi(L^1(G))''=\pi(G)''\subset B(\Hi_\pi)$. The Fourier space $A_\pi$ naturally identifies with the predual of $VN_\pi$ via the pairing $<u,T>\,=\sum_{n=1}^\infty\lla Tx_n,y_n\rra$, for $u\in B(G)$ and $T\in VN_\pi$, where $u=\sum_{n=1}^\infty\pi_{x_n,y_n}$ and $\sum_{n=1}^\infty\|x_n\|\|y_n\|<\infty$. In particular, this implies that the Fourier algebra $A(G)$ is the predual of the group von Neumann algebra $VN(G)=VN_\lambda$. For representations $\pi$ and $\sigma$ of $G$, the Fourier space $A_\pi$ is contained in $A_\sigma$ if and only if $\pi$ is quasi-contained in $\sigma$, i.e., if and only if $\pi$ is contained in some amplification $\sigma^{\oplus\alpha}$ of $\sigma$ for some cardinal $\alpha$, which occurs if and only if the map $\sigma(f)\mapsto \pi(f)$ for $f\in L^1(G)$ extends to a normal $*$-isomorphism from $VN_\sigma\to VN_\pi$. Hence, there is a one-to-one correspondence between group von Neumann algebras $VN_\pi$ of $G$ and Fourier spaces $A_\pi$.

Let $\mc S$ be a collection of representations of $G$. The {\it Fourier-Stieltjes} space $B_{\mc S}$ is defined to be the closure of $A_{\mc S}$ in the weak*-topology $\sigma(B(G),C^*(G))$. Since every Fourier space is realizable as a space $A_\pi$ for some representation $\pi$ of $G$, every Fourier-Stieltjes space is also realizable as $B_\sigma$ for some representation $\sigma$.

Let $\pi$ be a representation of $G$. Then the Fourier-Stieltjes space $B_\pi$ can be identified with the $C^*$-algebra $C^*_\pi:= \overline{\pi(L^1(G))}^{\|\cdot\|}\subset B(\Hi_\pi)$ via the pairing $<u,\pi(f)>\,=\int_G u(s)f(s)\,ds$. Let $\mc S$ and $\mc S'$ be two collections of representations of $G$. Then $B_{\mc S}\subset B_{\mc S'}$ if and only if $\sup_{\pi\in \mc S}\|\pi(f)\|\leq \sup_{\sigma\in \mc S'} \|\sigma(f)\|$ for every $f\in L^1(G)$, i.e., if and only if $\mc S$ is weakly contained in $\mc S'$. Hence, there is a one-to-one correspondence between group $C^*$-algebras $C^*_\pi$ of $G$ and Fourier-Stieltjes spaces $B_\pi$.

In this paper we study the Fourier and Fourier-Stieltjes spaces associated to the $L^p$-representations of a locally compact group $G$. These representations are defined in the next section.

\section{Preliminaries on $L^p$-representations and corresponding $C^*$-algebras}

The theory of $L^p$-representations and their corresponding $C^*$-algebras for discrete groups was recently developed by Nate Brown and Erik Guentner in \cite{bg}. This paper has inspired further work by a number of other authors (see \cite{br,j,klq,w,w2}). Though Brown and Guentner defined $L^p$-representations in the context of discrete groups, their definitions and basic results generalize immediately to our context of locally compact groups. Rather than making explicit notes of this, we will simply state their results in the context of locally compact groups.

Let $G$ be a locally compact group and $D$ a linear subspace of $C_b(G)$. A representation $\pi\fn G\to B(\Hi_\pi)$ is said to be a {\it $D$-representation} if there exists a dense subspace $\Hi_0$ of $\Hi_\pi$ so that $\pi_{x,x}\in D$ for every $x\in \Hi_0$. The following facts are noted in \cite{br} and \cite{bg}, and are easily checked:
\begin{itemize}
\item The $D$-representations are closed under arbitrary direct sums.
\item If $D$ is a subalgebra of $C_b(G)$, then the tensor product of two $D$-representations remains a $D$-representation.
\item If $D$ is an ideal of $C_b(G)$, then the tensor product of a $D$-representation with any representation is a $D$-representation.
\end{itemize}
For our purposes, we will be most interested in studying the case when $D=L^p(G)\cap C_b(G)$ for $p\in [1,\infty)$. In this case, the left regular representation $\lambda$ of $G$ is an $L^p$-representation since taking the dense subspace of $L^2(G)$ to be $C_c(G)$ clearly satisfies the required condition.

To each linear subspace $D$ of $C_b(G)$ define a $C^*$-seminorm $\|\cdot\|_D\fn L^1(G)\to [0,\infty)$ by
$$\|f\|_D=\sup\{\|\pi(f)\|: \pi\tn{ is a $D$-representation}\}.$$
The $C^*$-algebra $C^*_D(G)$ is defined to be the ``completion'' of $L^1(G)$ with respect to this $C^*$-seminorm. When $D=L^p(G)\cap B(G)$, we write $C^*_{L^p}(G)=C^*_D(G)$ and $\|\cdot\|_{L^p}=\|\cdot\|_D$. This process of building $C^*$-algebras was originally completed in the case when $D$ was an ideal of $\ell^\infty(\Gamma)$ of a discrete group $\Gamma$, and was called an {\it ideal completion} \cite{bg}. We note that in the case when $D=L^p$ for some $p\in[1,\infty)$, then $\|\cdot\|_{L^p}$ dominates the reduced $C^*$-norm since $\lambda$ is an $L^p$-representation. A fotiori $\|\cdot\|_{L^p}$ is a norm on $L^1(G)$ and the identity map on $L^1(G)$ extends to a quotient map from $C^*_{L^p}(G)$ onto $C^*_r(G)$.

In general, it is desirable that the space $D\subset C_b(G)$ used in this construction is translation invariant (under both left and right translation). Indeed, this guarantees that if $u$ is a positive definite function on $G$ which lies in $D$, then the GNS representation of $u$ is a $D$-representation \cite[Lemma 3.1]{bg} and, hence, $u$ extends to a positive linear functional on $C^*_D(G)$.

%In the literature, the cases when the subspace $D\subset C_b(G)$ is taken to be $L^p$ for some $1\leq p<\infty$ or $C_0(G)$ have been more heavily studied than another.
The subspaces $D$ of $C_b(G)$ which have been most heavily studied in the context of $D$-representations are $C_0(G)$ and $L^p(G)$.
Brown and Guentner recognized both these cases in their paper and developed much of the basic theory for the associated $C^*$-algebras in their original paper. In the case when $D=L^p$, Brown and Guentner demonstrated that $C^*_{L^p}(G)=C^*_r(G)$ for every $p\in [1,2]$ \cite[Proposition 2.11]{bg}, and that if $C^*_{L^p}(G)=C^*(G)$ for some $p\in [1,\infty)$ then $G$ is amenable \cite[Proposition 2.12]{bg}.

Brown and Guentner demonstrated that this construction can produce an intermediate $C^*$-algebra between the reduced and full by showing that $C^*_{\ell^p}(\F_d)\neq C^*_r(\F_d)$ for some $p\in (2,\infty)$ where $\F_d$ is the free group on $2\leq d<\infty$ generators \cite[Proposition 4.14]{bg}. Subsequently, Okayasu showed that the $C^*$-algebras $C^*_{\ell^p}(\F_d)$ are distinct for every $p\in [2,\infty)$ (this was also independently shown by both Higson and Ozawa).

Let $G$ be a locally compact group and $H$ an open subgroup of $G$ and suppose that $\pi\fn H\to B(\Hi)$ is an $L^p$-representation of $H$. Then $\mathrm{ind}_H^G \pi$ is an $L^p$-representation of $G$. We proved this for subgroups of discrete groups in \cite[Theorem 2.4]{w}, but the proof holds for any open subgroup of a locally compact group. Since every $L^p$-representation of $G$ clearly restricts to an $L^p$-representation of $H$, it follows that $\|\cdot\|_{L^p(G)}|_{L^1(H)}=\|\cdot\|_{L^p(H)}$. Hence, in the case when $\Gamma$ is a discrete group containing a copy of a noncommutative free group, it follows that the $C^*$-algebras $C^*_{\ell^p}(\Gamma)$ are distinct for every $p\in [2,\infty)$.

%Throughout the rest of the paper we will concern ourselves with $A_{L^p}(G)$ and $B_{L^p}(G)$ which are the Fourier and Fourier-Stieltjes spaces associated with these $L^p$-representations.

\section{Definition and basic properties of $L^p$-Fourier(-Stieltjes) spaces}

%As the study and use of $L^p$-representations becomes more prevalent in the literature, it becomes natural to study their associated Fourier and Fourier-Stieltjes spaces. Indeed, Brannan and Ruan defined and have begun studying these spaces in \cite{br} and we have used these spaces to produce intermediate $C^*$-tensor product norms in \cite{w}, however there is still much work to be done.

%We previously mentioned that the two types of $D$-representations which have been most heavily studied in the literature are $L^p$-representations and $C_0$-representations. For $C_0$-representations, the associated Fourier space has already been well studied and preceeds the paper of Brown and Guentner. In the liturature these spaces are known as the Rajchman algebra AND HAVE BEEN STUDIED BY. We find this to be further motivation for the study of these spaces. In this section we simply define and prove basic properties $L^p$-Fourier and Fourier-Stieltjes spaces.

Recall that the $C_0$-representations and the $L^p$-representations are the two most heavily studied types of $D$-representations in the literature. In \cite{br} Brannan and Ruan define the $D$-Fourier algebra $A_D(G)$ and $D$-Fourier-Stieltjes algebra $B_D(G)$ when $D$ is a subalgebra of $C_b(G)$. When $D=C_0(G)$, the $D$-Fourier algebra $A_D(G)$ is already well studied and is known in the literature as the Rajchman algebra. In contrast, very little has been done in regards to the $L^p$-Fourier and $L^p$-Fourier-Stieltjes algebras. We aim to help fill this gap with our paper. In this section we recall the definitions and prove some basic properties of these spaces.

Let $D$ be a linear subspace of $C_b(G)$. The $D$-Fourier space is defined to be
$$A_D(G)=A_D:=\{\pi_{x,y} : \pi\tn{ a $D$-representation, }x,y\in \Hi_\pi\}.$$
Similarly, the $D$-Fourier-Stieltjes space $B_D(G)=B_D$ is defined to be the closure of $A_D$ with respect to the weak*-topology $\sigma(B(G),C^*(G))$. When the subspace $D$ of $C_b(G)$ is a subalgebra (resp., ideal) of $C_b(G)$, then Brannan and Ruan noted that $A_D(G)$ and $B_D(G)$ are subalgebras (resp., ideals) of $B(G)$ \cite[Proposition 3.11]{br}. In these cases, we may call $A_D(G)$ and $B_D(G)$ the $D$-Fourier algebra and $D$-Fourier-Stieltjes algebra, respsectively. Note that since $L^p$ is an ideal in $C_b(G)$, $A_{L^p}(G)$ and $B_{L^p}(G)$ are ideals in $B(G)$.

Let $D$ be a subspace of $C_b(G)$. In section 2 we defined the Fourier space $A_{\mc S}$ and the Fourier-Stieltjes space $B_{\mc S}$ when $\mc S$ is a collection of representations of $G$. As an immediate consequence of the next proposition, we get that $A_D=A_{\mc S}$ and $B_D=B_{\mc S}$ when $\mc S$ is taken to be the collection of $D$-representations of $G$.

\begin{prop}
Let $D$ be a subspace of $C_b(G)$. Then $A_D$ is a closed translation invariant subspace of $B(G)$. Moreover,
$$\|u\|_{B(G)}=\inf\{\|x\|\|y\|: u=\pi_{x,y}\tn{ and }\pi\tn{ is a }D\tn{-representation}\}$$
and this infimum is attained for some $D$-representation $\pi$ and $x,y\in \Hi_\pi$.
\end{prop}

\begin{proof}
For every $u\in A_D$, choose a $D$-representation $\pi_u\fn G\to B(\Hi_u)$ so that $u=(\pi_u)_{x,y}$ for some $x,y\in \Hi_u$. Then $\pi:=\bigoplus_{u\in A_D} \pi_u$, being a direct sum of $D$-representations, is also a $D$-representation. Then $A_\pi\supset A_D$ since every element $u\in A_D$ is a coefficient function of $\pi$.

Now let $u\in A_\pi$. Then we can find sequences $\{x_n\},\{y_n\}$ in $\Hi_\pi$ so that $u=\sum_{n=1}^\infty \pi_{x_n,y_n}$ and $\|u\|=\sum_{n=1}^\infty \|x_n\|\|y_n\|$. Let $\tilde{\pi}\fn G\to B(\Hi_\pi^{\bigoplus\infty})$ be the infinite amplification $\infty\cdot \pi$. Then, since $\tilde{\pi}$ is a $D$-representation and $u=\tilde\pi_{(x_n),(y_n)}$, we arrive at the desired conclusions.
\end{proof}

These observations allow us to identify $B_D$ with the dual of $C^*_D(G)$.

\begin{prop}
Let $D$ be a subspace of $C_b(G)$. Then $B_D$ is identified with the dual space of $C^*_D(G)$ via the dual pairing $<u,f>\,=\int u(s)f(s)\,ds$ for $f\in L^1(G)$. 
\end{prop}

\begin{proof}
Let $\pi$ be the representation from the proof of the previous proposition and note that we demonstrated that $A_D=A_\pi$. Hence, it suffices to check that $C^*_D(G)=C^*_\pi$.

Since $\pi$ is a $D$-representation, $\|\pi(f)\|\leq \|f\|_D$ for every $f\in L^1(G)$. Now let $\sigma$ be a $D$-representation of $G$. Then $A_\sigma\subset A_D=A_\pi$ implies that $B_\sigma\subset B_\pi$ and, hence, that $\|\sigma(f)\|\leq \|\pi(f)\|$ for every $f\in L^1(G)$. Thus, $\|f\|_D=\|\pi(f)\|$ for every $f\in L^1(G)$.
\end{proof}

%So we get the following immediate corollary.
%
%\begin{corollary}
%$B_{D_1}=B_{D_2}$ if and only if $C^*_{D_1}(G)=C^*_{D_2}(G)$.
%\end{corollary}

Let $P(G)$ denote the set of positive definite functions on $G$. Then $A_D$ has a very nice description in terms of the linear span of positive definite functions when $D$ is a translation invariant subspace of $C_b(G)$.

\begin{prop}\label{linear combination}
Suppose that $D$ is a translation invariant subspace of $C_b(G)$. Then $A_D$ is the closed linear span of $P(G)\cap D$ in $B(G)$.
\end{prop}

\begin{proof}
Let $u\in P(G)\cap D$. Then, since the GNS representation of $u$ is a $D$-representation, $u$ is clearly in $A_D$. As $A_D$ is a closed subspace of $B(G)$, we conclude that $A_D$ contains the closed linear span of $P(G)\cap D$.

Now let $u\in A_D$. Then we can write $u=\pi_{x,y}$ for some $D$-representation $\pi$ of $G$ and $x,y\in\Hi_\pi$. Let $\Hi_0$ be a dense subspace of $\Hi_\pi$ so that $\pi_{z,z}\in D$ for every $z\in \Hi_0$ and choose sequences $\{x_n\}$, $\{y_n\}$ in $\Hi_0$ converging in norm to $x$ and $y$, respectively. Then
$$\pi_{x_n,y_n}=\sum_{k=0}^3 i^k\pi_{x_n+i^k y_n, x_n+i^k y_n}$$
converges to $u=\pi_{x,y}$ in norm. Hence, $A_D$ is the closed linear span of $P(G)\cap D$.
\end{proof}

For the remainder of this section, we will focus specifically on $L^p$-Fourier and Fourier-Stieltjes algebras. We begin by identifying cases when these spaces are familiar subspaces of $B(G)$.

\begin{prop}\label{basic}
Let $G$ be a locally compact group.
\begin{itemize}
\item[(i)] $A_{L^p}(G)=A(G)$ for every $p\in [1,2]$.
\item[(ii)] If $G$ is compact, then $A_{L^p}(G)=B(G)$ for every $p\in [1,\infty)$.
\item[(iii)] If $G$ is amenable, then $B_{L^p}(G)=B(G)$ for every $p\in [1,\infty)$.
\item[(iv)] If $B_{L^p}(G)=B(G)$ for some $p\in [1,\infty)$, then $G$ is amenable.
\end{itemize}
\end{prop}

\begin{proof}
(i) Recall that the Fourier algebra $A(G)$ is both the closed linear span of $P(G)\cap C_c(G)$ and of $P(G)\cap L^2(G)$ \cite[Proposition 3.4]{e}. Since $P(G)\cap C_c(G)\subset P(G)\cap L^p(G)\subset P(G)\cap L^2(G)$ for every $p\in [1,2]$, we arrive at the desired conclusion.

(ii) Let $\pi$ be a representation of $G$ and $x\in\Hi_\pi$. Then $\pi_{x,x}$ is bounded in uniform norm by $\|x\|^2$. Since $x\in\Hi_\pi$ was arbitrary, we conclude that every representation $\pi$ of $G$ is an $L^p$-representation and, hence, that $A_{L^p}(G)=B(G)$.

(iii) Since $G$ is amenable, $C^*_r(G)=C^*(G)$. Hence, the reduced Fourier-Stieltjes algebra $B_\lambda=B(G)$. Since $A_{L^p}(G)\supset A(G)$, we conclude that $B_{L^p}(G)$ must also be all of $B(G)$.

(iv) If $B_{L^p}(G)=B(G)$, then $C^*_{L^p}(G)=C^*(G)$ and, hence, $G$ is amenable \cite{bg}.
\end{proof}

\begin{prop}\label{Holder}
Let $G$ be a locally compact group and $p,q,r\in [1,\infty)$ be such that $\frac{1}{p}+\frac{1}{q}=\frac{1}{r}$. Then $uv\in A_{L^r}(G)$ for every $u\in A_{L^p}(G)$ and $v\in A_{L^q}(G)$. Similarly, $uv\in B_{L^r}(G)$ for all $u\in B_{L^p}(G)$ and $v\in B_{L^q}(G)$.
\end{prop}

\begin{proof}
Let $u\in A_{L^p}(G)$ and $v\in A_{L^q}(G)$. By Proposition \ref{linear combination}, we can approximate $u$ and $v$ well in norm by linear combinations $a_1 u_1+\ldots+a_n u_n$ and $b_1 v_1+\ldots+b_m v_m$ of positive definite elements in $L^p(G)$ and $L^q(G)$, respectively. Then the product
$$ \sum_{i,j} a_ib_j u_i v_j$$
is a linear combination of elements in $P(G)\cap L^r(G)$ approximating $uv$ well in norm. Hence, $uv\in A_{L^r}(G)$ by Proposition \ref{linear combination}.

Note that since multiplication in $B(G)$ is separately weak*-weak* continuous, it follows that $uv\in B_{L^r}(G)$ for all $u\in B_{L^p}(G)$ and $v\in B_{L^q}(G)$.
\end{proof}

\begin{prop} \label{open}
Suppose $H$ is an open subgroup of a locally compact group $G$ and $1\leq p<\infty$. Then $A_{L^p}(H)=A_{L^p}(G)|_H$ and $B_{L^p}(H)=B_{L^p}(G)|_H$.
\end{prop}

\begin{proof}
The first part of the statement is deduced by similar reasoning as used in the previous section. Indeed, the equality $A_{L^p}(G)|_H=A_{L^p}(H)$ follows from the definition of $L^p$-Fourier spaces since $\pi|_H$ is an $L^p$-representation for every $L^p$-representation $\pi$ of $G$, and $\mathrm{ind}_H^G \sigma$ is an $L^p$-representation for every $L^p$-representation $\sigma$ of $H$.

We now proceed to prove the second part of the statement. Notice that since $H$ is an open subgroup of $G$, $L^1(H)$ embeds naturally into $L^1(G)$. A similar argument as above shows that this extends to a natural embedding of $C^*_{L^p}(H)$ into $C^*_{L^p}(G)$. Let $u\in C^*_{L^p}(H)^*=B_{L^p}(H)$. Then, by the Hahn-Banach theorem, there is an element $\widetilde u\in C^*_{L^p}(G)^*=B_{L^p}(G)$ extending $u$ as a linear functional. Then for $f\in L^1(H)\subset L^1(G)$,
$$\int_H \widetilde u(s)f(s)\,ds=\int_G \widetilde u(s)f(s)\,ds=\,<\widetilde u, f>\,=\,<u,f>\,=\int_H u(s)f(s)\,ds.$$
So $u=\widetilde u|_H$ almost everywhere. Since $u$ and $\widetilde u$ are each continuous functions, this implies that $u=\widetilde u|_H$ and, hence, that $B_{L^p}(H)\subset B_{L^p}(G)|_H$. A similar but simpler argument shows that $B_{L^p}(H)\supset B_{L^p}(G)|_H$.
\end{proof}

The above proposition can fail, even for $A_{L^p}$, when $H$ is a non-open closed subgroup of $G$. This is demonstrated by example in Remark \ref{Herz fails}.

We finish this section by giving a first class of examples of groups $G$ which show that $A_{L^p}(G)$ and $B_{L^p}(G)$ are interesting subspaces of $B(G)$ for $2<p<\infty$.

\begin{prop}
Let $\Gamma$ be a discrete group containing a copy of a noncommutative free group. Then $B_{\ell^p}(\Gamma)$ is distinct for every $p\in [2,\infty)$. Hence, $A_{\ell^p}(\Gamma)$ is also distinct for every $p\in [2,\infty)$.
\end{prop}

\begin{proof}
The first statement is immediate from previous comments since $C^*_{\ell^p}(\Gamma)$ is distinct for each $p\in [2,\infty)$. The second statement follows from the first since $B_{\ell^p}(\Gamma)$ is the weak*-closure of $A_{\ell^p}(\Gamma)$.
\end{proof}

\section{$L^p$-Fourier algebras of Abelian groups}

In this section we show that the algebras $A_{L^p}(G)$ are distinct for every $p\in [2,\infty)$ when $G$ is a noncompact locally compact abelian group. We will later see that this phenomena does not generalize to the setting of general noncompact locally compact groups which shows that we are required to use tools from commutative harmonic analysis. Before entering into proofs, we provide a brief summary of the tools of which we make use and refer the reader to Graham and McGehee's book \cite{gm} for more details.

Let $\Gamma$ be a discrete abelian group. A subset $\Theta$ of $\Gamma$ is said to be {\it dissociate}  if every element $\omega\in \Gamma$ can be written in at most one way as a product
$$\omega=\prod_{j=1}^n\theta_j^{\epsilon_j}$$
where $\theta_1,\ldots,\theta_n\in\Theta$ are distinct elments, $\epsilon_j=\pm 1$ if $\theta_j\neq 1$, and $\epsilon_j=1$ if $\theta_j^2=1$. As an example, if $\Gamma=\Z$ then the set $\{3^j : j\geq 1\}$ is dissociate. As in the case of the integers, every infinite discrete abelian group admits an infinite dissociate set.

Let $G$ be a compact abelian group with normalized haar measure and $\Gamma=\widehat{G}$ be the dual group of $G$. If $\gamma$ is a group element of $\Gamma$ such that  $\gamma^2\neq 1$ and $a(\gamma)$ is a constant with $|a(\gamma)|\leq 1/2$, then the trigonometric polynomial
$$q_\gamma:=1+a(\gamma)\gamma+\overline{a(\gamma)}\overline{\gamma}$$
is a positive function on $G$ with $\|q_\gamma\|_1=1$. Similarly, if $\gamma\in\Gamma\backslash\{1\}$ has the property that $\gamma^2=1$ and $0\leq a(\gamma)<1$, then $q_\gamma:=1+a(\gamma)\gamma$ is a positive function which integrates over $G$ to 1. We will consider weak* limits of products of polynomials of this type.

Let $\Theta\subset \Gamma$ be a dissociate set. To each $\theta\in \Theta$ assign a value $a(\theta)\in\C$ with the imposed restrictions from above. For each finite subset $\Phi\subset\Theta$ define $P_{\Phi}=\prod_{\theta\in\Phi} q_\theta$. This being a product of positive functions is a positive function on $G$ with Fourier transform
$$ \widehat{P_\Phi}(\gamma)=\left\{\begin{array}{c l}
\prod_{\theta\in\Phi} a(\theta)^{(\epsilon_\theta)}, & \gamma=\prod_{\theta\in \Phi}\theta^{\epsilon_\theta} \\
0, & \tn{otherwise}
\end{array}\right.$$
where $\epsilon_\theta$ range over $\{-1,0,1\}$ and
$$a(\theta)^{(\epsilon_\theta)}=\left\{\begin{array}{c l}
1, & \epsilon_\theta=0 \\
a(\theta), & \epsilon_\theta=1 \\
\overline{a(\theta)}, & \epsilon_\theta=-1.
\end{array}\right.$$
It follows that as $\Phi\nearrow\Theta$, $P_{\Phi}$ converges weak* to a measure $\mu$ on $G$ where
$$ \widehat{\mu} (\gamma)= \left\{\begin{array}{c l}
\prod_{\theta\in\Theta} a(\theta)^{(\epsilon_\theta)}, & \gamma=\prod_{\theta\in \Theta}\theta^{\epsilon_\theta},\  \epsilon_\theta=0\tn{ for all but finitely many }\theta \\
0, & \tn{otherwise}
\end{array}\right. $$
The measure $\mu$ is said to be based on $\Theta$ and $a$. This method of constructing measures is called the Riesz Product construction and the set of all such constructions is denoted $R(G)$.

In 1959 Zygmund \cite{z} proved that a measure $\mu\in R(\T)$ based on $\Theta$ and $a$ is an element of $L^1(\T)$ if and only if $a\in\ell^2(\Theta)$. This result was extended to all compact abelian groups $G$ by Hewitt and Zuckerman \cite{hz} in 1966. If $\Gamma$ is the dual of a compact group $G$ and $\mu\in R(G)$ is based on $\Theta$ and $a$, then $\widehat{\mu}\in A(\Gamma)=A_{\ell^2}(\Gamma)$ if and only if $a\in\ell^2(\Theta)$. We will demonstrate that the analogue of this theorem holds when $2$ is replaced with $p$ for $2\leq p<\infty$. Towards this goal, we begin by proving an elementary lemma which is surely known but we include for convenience to the reader and lack of a reference.

\begin{lemma}\label{technical lemma}
Suppose that $0<\alpha<1$ and $\{x_n\}$ is a bounded sequence  but $\{x_n\}\not\in \ell^p$. Then there exists a bounded sequence $\{y_n\}$ so that $\{x_ny_n\}\in \ell^p$ but $\{x_ny_n^\alpha\}\not\in\ell^p$.
\end{lemma}

\begin{proof}
Clearly it suffices to consider the case when $p=1$. We first focus our attention to the case when $\{x_n\}\in c_0$. Then we can choose mutually disjoint subsets $I_1,I_2,\ldots$ of $\N$ so that $\sum_{n\in I_k}|x_n|=1$ for each $k$. Define
$$ y_{n}=\left\{\begin{array}{c l}
k^{-\frac{1}{\alpha}} & \tn{if }n\in I_k \\
0 & \tn{otherwise}
\end{array}\right.$$
Then
%\begin{eqnarray*}
%\sum_{n\in\N} |x_ny_{n}| &=& \sum_k \sum_{n\in I_k} |x_ny_{n}| \\
%&=& \sum_k \sum_{n\in I_k} |x_n| k^{-\frac{1}{\alpha}} \\
%&=& \sum_k k^{-\frac{1}{\alpha}} <\infty,
%\end{eqnarray*}
$$\sum_{n\in\N} |x_ny_{n}| = \sum_k \sum_{n\in I_k} |x_ny_{n}| = \sum_k \sum_{n\in I_k} |x_n| k^{-\frac{1}{\alpha}} = \sum_k k^{-\frac{1}{\alpha}} <\infty,$$
but
%\begin{eqnarray*}
%\sum_{n\in\N} |x_ny_n^\alpha| &=& \sum_k \sum_{n\in I_k} |x_ny_{n}^\alpha| \\
%&=& \sum_k \sum_{n\in I_k} |x_n| k^{-1} \\
%&=& \sum_k k^{-1} =\infty.
%\end{eqnarray*}
$$\sum_{n\in\N} |x_ny_n^\alpha| = \sum_k \sum_{n\in I_k} |x_ny_{n}^\alpha| = \sum_k \sum_{n\in I_k} |x_n| k^{-1} = \sum_k k^{-1} =\infty.$$

Now assume that $\limsup |x_n|>0$. Then we can find $\delta>0$ and a subsequence $\{x_{n_k}\}$ so that $
|x_{n_k}|\geq \delta$ for every $k$. Defining
$$y_{n}=\left\{\begin{array}{c l}
k^{-\frac{1}{\alpha}} & \tn{if }n=n_k \\
0 & \tn{otherwise.}
\end{array}\right.$$
gives the desired result.

%We now deduce that it is possible to build a sequence $\{y_n\}$ which works for every $\alpha\in (0,1)$. First observe that if $0<\alpha_1<\alpha_2<1$, then $x_ny_{n,\alpha_2}^{\alpha_1}>x_ny_{n,\alpha_2}^{\alpha_2}$ implies that
%$$\sum_{n=1}^\infty x_ny_{n,\alpha_2}^{\alpha_1}.$$
%Let $\{\alpha_k\}\subset (0,1)$ be a sequence increasing to 1 and let $s_k=\sum_{n=1}^\infty x_n y_{n,\alpha_k}<\infty$. Define $$ y_n=\sum_{k=1}^\infty \frac{1}{2^k s_k} y_{n,\alpha_k}$$
%for each $n\in\N$. Then $\{x_ny_n\}\in\ell^1$ and, by our above observation, $\{x_ny_n^{\alpha}\}\not\in\ell^1$ for each $\alpha\in (0,1)$.
\end{proof}

We are now prepared to prove the main result of this section in the case of discrete groups.

\begin{theorem}\label{discrete abelian}
Let $G$ be a compact abelian group with dual group $\Gamma$ and $\mu\in R(G)$ be based on $\Theta$ and $a$. Then $\widehat{\mu}\in A_{\ell^p}(\Gamma)$ if and only if $a\in \ell^p(\Theta)$.
\end{theorem}

\begin{proof}
First we suppose that $ \sum_{\theta\in\Theta} \big|a(\theta)\big|^p<\infty$ and let
$$\Omega(\Theta)=\{\theta_1\cdots \theta_n \mid \theta_1^{\epsilon_1},\ldots,\theta_n^{\epsilon_n}\in\Theta \tn{ distinct, }\epsilon_1,\ldots,\epsilon_n=\pm 1,\, n\geq 0\}.$$ Then
\begin{eqnarray*}
\|\widehat{\mu}\|_p^p &=& \sum_{\omega\in\Omega(\Theta)}\big|\widehat{\mu}(\omega)\big|^p \\
%&\leq & 1+\sum_{n=1}^\infty 2^n\sum_{\{\theta_1,\ldots,\theta_n\}\subset \Theta\atop|\{\theta_1,\ldots,\theta_n\}|=n} \big|a(\theta_1)\cdots a(\theta_n)\big|^p\\
&\leq & 1+\sum_{n=1}^\infty \sum_{\Phi\subset \Theta\atop |\Phi|=n}2^n\prod_{\theta\in\Phi}|a(\theta)|^p\\
%&\leq & 1+\sum_{n=1}^\infty 2^n\sum_{\theta_1,\ldots,\theta_n\in\Theta}\frac{\big|a(\theta_1)\cdots a(\theta_n)\big|^p}{n!} \\
&\leq & 1+\sum_{n=1}^\infty \frac{2^n}{n!} \bigg(\sum_{\theta\in\Theta}\big|a(\theta)\big|^p\bigg)^n \\
&=& \mathrm{exp}\bigg\{2\sum_{\theta\in\Theta}\big|a(\theta)\big|^p\bigg\}<\infty.
\end{eqnarray*}
Hence, $\widehat{\mu}\in A_{\ell^p}(\Gamma)$.

Now suppose that $ \sum_{\theta\in\Theta} \big|a(\theta)\big|^p=\infty$ but $\widehat{\mu}\in A_{\ell^p}(\Gamma)$. Hewitt and Zuckerman showed this is not possible for $p=2$, so we will assume without loss of generality that $p>2$. Choose a sequence $\{b(\theta)\}\in \ell^\infty(\Theta)$  with $\|b\|_\infty\leq 1$ so that $\{a(\theta)b(\theta)\}\in \ell^p(\Theta)$ but $\{a(\theta)b(\theta)^{\alpha}\}\not\in \ell^p(\Theta)$ for $\alpha=\frac{p-2}{2p}$. Let $\nu\in R(G)$ be based on $\Theta$ and $c:=\{(a(\theta)b(\theta))^{\frac{p-2}{2}}\}$. Define $q=\frac{2p}{p-2}$ (this is chosen so that $1/p+1/q=1/2$). Then
$$ \sum_{\theta\in\Theta} \big|c(\theta)\big|^q = \sum_{\theta\in\Theta}\big|a(\theta)b(\theta)\big|^p<\infty $$
implies that $\widehat{\nu}\in \ell^q(\Theta)$. So $\widehat{\mu}\cdot\widehat{\nu}\in A_{\ell^2}(\Gamma)=A(\Gamma)$ by Proposition \ref{Holder}. Observe that $\mu *\nu$ is the element in $R(G)$ generated by $\Theta$ and $a\cdot c$. So $a\cdot c\in\ell^2(\Theta)$. But, by our assumption on $b$,
$$\sum_{\theta\in\Theta}\big|a(\theta)c(\theta)\big|^2=\sum_{\theta\in\Theta}\big|a(\theta)b(\theta)^{\frac{p-2}{2p}}\big|^p=\infty,$$
a contradiction. Therefore, $\widehat{\mu}\in A_{\ell^p}(\Gamma)$ iff $a\in\ell^p(\Theta)$.
\end{proof}

\begin{corollary}
Let $\Gamma$ be an infinite discrete Abelian group. The subspaces $A_{\ell^p}(\Gamma)$ of $B(\Gamma)$ are distinct for every $p\in [2,\infty)$.
\end{corollary}

Our next step is show that $A_{L^p}(G)$ is distinct for each $2\leq p<\infty$ for another class of locally compact abelian groups $G$.

Suppose $\Gamma$ is a lattice in a locally compact abelian group $G$. Further, suppose that $v\in A(G)$ is a normalized positive definite function with the property that $\mathrm{supp}\,v\cap \Gamma=\{e\}$ and $(s+\mathrm{supp}\,v)\cap \Gamma$ is finite for every $s\in G$. Then the map $J=J_v$ from $B(\Gamma)$ into $B(G)$ defined by
$$ Ju(s)=\sum_{\xi\in \Gamma} u(\xi)v(s-\xi)$$
is a well defined isometry with the following properties: \cite[Theorem A.7.1]{gm}
\begin{itemize}
\item[(i)] $Ju\in P(G)$ if and only if $u\in P(\Gamma)$;
\item[(ii)] $Ju\in A(G)$ if and only if $u\in A(\Gamma)$.
\end{itemize}

\begin{lemma}\label{n>0}
Let $G=\R^n\times K$ for some compact abelian group $K$ and $n\geq 1$. Choose a normalized $v\in A(G)\cap P(G)$ so that $\mathrm{supp}\,v\subset [-1/3,1/3]^n\times K$ and suppose $\mu\in R(\Z^n)$ is based on $\Theta$ and $a$. Then $J_v\widehat{\mu}\in A_{L^p}(\R^n)$ if and only if $a\in\ell^p(\Theta)$.
\end{lemma}

\begin{proof}
We leave it as an exercise to the reader to check that if $\mathrm{supp}\,v\subset [-1/3,1/3]^n\times K$, then $\mathrm{supp}\,v\cap \Z^n=\{e\}$ and $(s+\mathrm{supp}\,v)\cap \Z^n$ is finite for every $s\in G$.

Let $u\in P(\Gamma)\cap \ell^p(\Gamma)$. For $(x_1,\ldots,x_n,k)\in \R^n\times K$, write $x_i=m_i+y_i$ for some $m_i\in\Z$ and $y_i\in [-1/2,1/2]$ ($1\leq i\leq n$). Then
\begin{eqnarray*}
 J_vu(x_1,\ldots,x_n,k) &=& \sum_{(\ell_1,\ldots,\ell_n)\in\Z^n} u(\ell_1,\ldots,\ell_n)v(m_1+y_1-\ell_1,\ldots,m_n+y_1-\ell_n,k) \\
 &=& u(m_1,\ldots, m_n)v(y_1,\ldots,y_n,k).
\end{eqnarray*}
For each $m_1,\ldots, m_n\in \Z^n$, define $M_{m_1,\ldots,m_n}=[m_1-1/2,m_1+1/2]\times\cdots\times [m_n-1/2,m_n+1/2]\times K$. Then 
\begin{eqnarray*}
&&\int_G |J_vu|^p\\
&=& \sum_{(m_1,\ldots,m_n)\in\Z^n}\int_{M_{m_1,\ldots,m_n}}|J_vu|^p \\
&=& \sum_{(m_1,\ldots,m_n)\in\Z^n}\int_{[-1/2,1/2]^n\times K}|u(m_1,\ldots,m_n)v(y_1,\ldots,y_n,k)|^p\,d(y_1,\ldots,y_n,k)\\
&=& \|u\|_p^p\int_{[-1/2,1/2]^n\times K}|v(y_1,\ldots,y_n,k)|^p\,d(y_1,\ldots,y_n,k) <\infty.
\end{eqnarray*}
Hence, $J_v u\in L^p(G)$. As $J$ is an isometry mapping $P(\Gamma)$ into $P(G)$ and $A_{\ell^p}(\Gamma)$ is the closed linear span of $P(\Gamma)\cap \ell^p(\Gamma)$, it follows that $J_v$ maps $A_{\ell^p}(\Gamma)$ into $A_{L^p}(G)$.

Let $\mu\in R(\Z^n)$ be based on $\Theta$ and $a$, and suppose that $a\not\in \ell^p(\Theta)$. Let $c$ be chosen as in the proof of Theorem \ref{discrete abelian} and $\nu\in R(\Z^n)$ be based on $\Theta$ and $c$. Then $\widehat{\nu}\in A_{\ell^q}(\Gamma)$ and, hence, $J_v \widehat{\nu}\in A_{L^q}(G)$ where $q$ satisfies $1/p+1/q=1/2$. For $m_1,\ldots,m_n\in\Z$, $y_1,\ldots,y_n\in [-1/2,1/2]$ and $k\in K$,
\begin{eqnarray*}
&& J_v\widehat{\mu}(m_1+y_1,\ldots,m_n+y_n,k)J_v\widehat{\nu}(m_1+y_1,\ldots,m_n+y_n,k)\cdot \\
&=& \widehat{\mu}(m_1,\ldots,m_n)\widehat\nu(m_1,\ldots, m_n)v(y_1,\ldots,y_n,k)^2 \\
&=&J_{v^2} \widehat{\mu *\nu}.
\end{eqnarray*}
Since $v^2$ is a positive definite function with support contained in $[-1/3,1/3]^n\times K$, $J_v\widehat{\mu*\nu}\in A(G)$ if and only if $\widehat{\mu*\nu}\in A(G)$. But $\mu*\nu$ is the element of $R(G)$ based on $\Theta$ and $a\cdot c$ and, as in the proof of Theorem \ref{discrete abelian}, $a\cdot c\not\in \ell^2(\Theta)$. So $\widehat{\mu*\nu}\not\in A(\Gamma)$ and, hence, $J_v\widehat{\mu}\cdot J_v \widehat{\nu}$ is not in $A(G)$. It follows that $J_v\widehat{\mu}\not\in A_{L^p}(G)$.
\end{proof}

\begin{corollary}
$A_{L^p}(G)$ is distinct for every $p\in [2,\infty)$ when $G=\R^n\times K$ where $K$ is some compact abelian group and $n\geq 1$.
\end{corollary}

\begin{proof}
It suffices to check that there is a nonzero positive definite function $v$ whose support is contained in $[-1/3,1/3]^n\times K$. Observe that
$$ \omega(x):=\chi_{[-1/6,1/6]}*\chi_{[-1/6,1/6]}(x)=\left\{\begin{array}{c l}
1-3|x|, & |x|\leq 1/3\\
0, & \tn{otherwise}
\end{array}\right.$$
is a positive definite function on $\R$ with support contained in $[-1/3,1/3]$. Taking $v=\omega\times\cdots\times\omega\times 1_K$ clearly does the trick.
\end{proof}

%\begin{proposition}
%Suppose $H$ is a subgroup of a locally compact group $G$, and $\mu\in M(H)$ a probability measure. For $\varphi\in B(G)$ define $\widetilde\varphi\fn G\to \C$ by
%$$ \widetilde\varphi(s)=\int\varphi(s\xi)d\mu(\xi). $$
%Then $\varphi\mapsto\widetilde\varphi$ is a contractive map on $B(G)$.
%\end{proposition}
%
%\begin{proof}
%Let $\varphi\in B(G)$ have norm 1, fix $\epsilon>0$, and choose a compact subset $K\subset H$ so that $\mu(K)>1-\epsilon$. Next, choose a symmetric open neighbourhood $U$ of the identity so that $s^{-1t}\in U$ implies that $|\varphi(s)-\varphi(t)|<\epsilon$ (this is possible since $\varphi$ is uniformly continuous).Find a disjoint cover $\{A_1,\ldots, A_{n-1}\}$ for $K$ such that $s^{-1}t\in U$ for every $s,t\in A_i$, $i=1,\ldots, n-1$. Define $A_n= H\backslash\bigcup_{i=1}^{n-1}A_i$. Then for any choice of $\xi_i\in A_i$,
%$$\left|\widetilde\varphi(s)-\sum_{i=1}^n\mu(A_i)\varphi(s\xi_i)\right|<2\epsilon$$
%for every $s\in G$. Note that $\sum_{i=1}^n\mu(A_i)\varphi(s\xi_i)=\sum_{i=1}^n\mu(A_i)\varphi_{\xi_i}(s)$ and $\sum_{i=1}^n\mu(A_i)\varphi_{\xi_i}$ is an element of $B(G)$ with norm at most 1. Hence, it follows from Raikov's theorem that $\widetilde\varphi\in B(G)$ with $\|\widetilde\varphi\|\leq 1$.
%\end{proof}

We now prove one last lemma before we show that $A_{L^p}(G)$ is distinct for each $p\in [2,\infty)$ when $G$ is any noncompact locally compact abelian group.

%\begin{lemma}
%Suppose $K$ is a normal compact subgroup of a locally compact group $G$. For every $\varphi\in B(G)$, define $\widetilde\varphi\fn G\to \C$ by $\widetilde\varphi(s)=\int \varphi(s\xi)dm_K(\xi).$ Then $\varphi\mapsto\widetilde\varphi$ is a contractive map on $B(G)$ with the property that $\widetilde\varphi\in P(G)$ whenever $\varphi\in P(G)$.
%\end{lemma}
%
%\begin{proof}
%Let $\varpi$ be the universal representation of $G$. Then $\widetilde\varphi=\varpi(m_K)\cdot\varphi\cdot \varpi(m_K)$.
%\end{proof}
%
%\begin{theorem}[Eymard]
%Let $\pi\fn G_1\to G$ be a continuous homomorphism with the property that $\sigma(G_1)$ is dense in $G$. For functions $f$ on $G$, we will let $j(f):=f\circ\sigma$. Then $j$ is an isometric map from $B(G)$ to $B(G_1)$ with the property that
%$$ j(B(G))=B(G_1)\cap j(C(G)). $$
%\end{theorem}

\begin{lemma}\label{n=0}
Suppose $K$ is a compact subgroup of a locally compact group $G$. Then
$$A_{L^p}(G:K):=\{u\in A_{L^p}(G) : u(sk)=u(s)\tn{ for all } s\in G, k\in K\}$$
is isometrically isomorphic to $A_{L^p}(G/K)$.
\end{lemma}

\begin{proof}
Let $m_K$ denote the normalized Haar measure for $K$ and note that $m_K$ is a central idempotent measure. Denote the universal representation of $G$ by $\varpi$ and define $p_K=\varpi(m_K)$. Observe that if $\pi$ is a representation of $G$, then $p_K\pi$ is constant on cosets of $K$ and, hence, defines a representation $\pi_K\fn G/K\to \mc U(\Hi_\pi)$ by $\pi_K(sK)=p_K\pi(s)$ for $s\in G$.

Suppose $\pi$ is an $L^p$-representation of $G$ and $\Hi_0$ is a dense subspace of $\Hi_\pi$ so that $\pi_{x,x}\in L^p(G)$ for all $x\in\Hi_0$. Let $q\fn G\to G/K$ be the canonical quotient map. Then
$$ (\pi_{K})_{p_Kx,p_Kx}\circ q= \lla p_K\pi(\cdot)x,x\rra= m_K* \pi_{x,x} \in L^p(G)$$
for all $x\in \Hi_0$. Since $m_K*L^p(G)\cong L^p(G/K)$, it follows that $\pi_K$ is an application of $G/K$.

Conversely, suppose that $\widetilde\pi$ is an $L^p$-representation of $G/K$. Then Weyl's integral formula implies that $\widetilde\pi\circ q$ is an $L^p$-representation of $G$. Furthermore, $m_K*(\widetilde\pi\circ q)_{x,y}=(\widetilde\pi\circ q)_{x,y}$ for all $x,y\in \Hi_{\widetilde \pi}$.
\end{proof}

We thank Nico Spronk for pointing out this previous lemma to us, which has allowed for cleaner arguments throughout the paper.

\begin{theorem}\label{abelian}
Let $G$ be a noncompact locally compact abelian group. Then $A_{L^p}(G)$ is distinct for every $p\in [2,\infty)$.
\end{theorem}

\begin{proof}
By the structure theorem for locally compact abelian groups, $G$ has an open subgroup of the form $\R^n\times K$ where $n\geq 0$ and $K$ is compact. If $n>0$, then the result follows from Lemma \ref{n>0}. Otherwise, it follows from Lemma \ref{n=0} that $A_{L^p}(\R^n\times K)$ is distinct for every $p\in [2,\infty)$ and, hence, $A_{L^p}(G)$ is distinct for every $p\in [2,\infty)$ by Proposition \ref{open}.
\end{proof}

We finish this section by showing that this same phenomenon which occurs for abelian groups also occurs in almost connected SIN groups.

\begin{theorem}
Let $G$ be a noncompact almost connected SIN group. Then $A_{L^p}(G)$ is distinct for every $p\in [2,\infty)$.
\end{theorem}

\begin{proof}
By the structure theorem for almost connected SIN groups, $G$ contains an open subgroup of finite index which is of the form $\R^n\times K$ for some $n\geq 0$ and compact group $K$ \cite{palm}. Then, since $G$  is noncompact, it is necessarily the case that $n\geq 1$. So it suffices to check this for groups $G$ of the form $\R^n\times K$ for some $n\geq 1$. As this follows from Lemma \ref{n=0}, we conclude that $A_{L^p}(G)$ is distinct for all $p\in [2,\infty)$.
\end{proof}

\section{The structure of $L^p$-Fourier(-Stieltjes) algebras}\label{structure}

In this section investigate the structural properties of the $L^p$-Fourier and Fourier-Stieltjes algebras with an emphasis on the former. Similar to the Fourier algebra, we find that the $L^p$-Fourier algebra completely determines the group. However, armed with our knowledge of these spaces in the cases when $G$ is either an abelian locally compact group or a discrete group containing a copy of a noncommutative free group, we observe that many nice properties which hold for Fourier algebras fail for $L^p$-Fourier algebras. We begin this section by determining the spectrum of the $L^p$-Fourier algebras.

\begin{prop}
Let $G$ be a locally compact group. Then the spectrum of $A_{L^p}(G)$ is $G$, where we identify elements of $G$ with their point evaluations.
\end{prop}

\begin{proof}
Clearly we have that $G\subset \sigma(A_{L^p}(G))$, so it suffices to check that $\sigma(A_{L^p}(G))\subset G$.
Let $\chi\in \sigma(A_{L^p}(G))$ and choose an integer $n$ so that $p/n\leq 2$. Then, since $u^n$ is in $A(G)$ for every $u\in A_{L^p}(G)$, there exists $s\in G$ so that $<\chi,u^n>\,=u(s)^n$ for all $u\in A_{L^p}(G)$ \cite[Th\'eor\'eme 3.34]{e}. As $<\chi,u><\chi,u^n>\,=\,<\chi, u^{n+1}>\,=u(s)^{n+1}$, it follows that $\chi$ is evaluation at $s$. Hence, we conclude that $\sigma(A_{L^p}(G))=G$.
\end{proof}

Recall that a linear functional $D$ on a Banach algebra $\mc A$ is said to be a point derivation if there exists some multiplicative linear functional $\chi$ on $\mc A$ so that $D(ab)=\chi(a)D(b)+D(a)\chi(b)$ for all $a,b\in \mc A$. The existence of nonzero point derivations is an obstruction to the (operator) weak amenability of $\mc A$. Since the Fourier algebra is always operator weakly amenable \cite{sp} (see also \cite{sa}), the Fourier algebra does not admit any nonzero point derivations. As a corollary to the above proposition, we show that the $L^p$-Fourier-Stieltjes algebras admit no nonzero point derivations either. This corollary was pointed out to us by Nico Spronk.

\begin{corollary}
Let $G$ be a locally compact group and $p\in [1,\infty)$. Then $A_{L^p}(G)$ does not admit any nonzero point derivations.
\end{corollary}

\begin{proof}
Suppose that $A_{L^p}(G)$ admits a nonzero point derivation $D$ and choose a multiplicative linear functional $\chi$ on $A_{L^p}(G)$ so that $D(uv)=D(u)\chi(v)+D(v)\chi(u)$ for all $u,v\in A_{L^p}(G)$. By the above proposition, $\chi$ is the point evaluation functional at some point $s\in G$. Choose $u\in A_{L^p}(G)$ and $v\in A(G)$ so that $D(u)\neq 0$ and $v(s)\neq 0$. Then
$$ D(uv)=D(u)\chi(v)+\chi(u)D(v)=v(s)D(u)\neq 0$$
since $v\in A(G)$ implies that $D(v)= 0$. But since $A(G)$ is an ideal in $B(G)$ and $A(G)$ admits no nonzero point derivations, we must have that $D(uv)=0$. This contradicts the above calculation and, therefore, we conclude that $A_{L^p}(G)$ admits no point derivations.
\end{proof}

One of the most coveted properties of the Fourier algebra $A(G)$ is that it completely determines the underlying locally compact group $G$. We now show that the analogue of this theorem holds for $A_{L^p}(G)$. The proof is similar to that given by Martin Walter and we refer the reader to his original paper \cite{walt} for most of the details.

\begin{theorem}
Let $G_1$ and $G_2$ be locally compact groups and suppose $A_{L^p}(G_1)$ is isometrically isomorphic to $A_{L^q}(G_2)$ as Banach algebras for some $p,q\in[2,\infty)$. Then $G_1$ is homeomorphically isomorphic to $G_2$.
\end{theorem}

\begin{proof}
Most of this proof is identical to that given by Walter and a careful read of his paper reveals that the only detail that is left to be verified is that the identification of $G$ with $\sigma(A_{L^p}(G))$ is a homeomorphic one when $\sigma(A_{L^p}(G))$ is equipped with the weak*-topology.

Let $VN_{L^p}(G)$ be the von Neumann algebra dual to $A_{L^p}(G)$ described in Section \ref{spaces}. Then the canonical embedding of $G$ into $VN_{L^p}(G)$ is continuous in the weak*-topologies. Denote this map by $\rho$. Then, since $A(G)$ is contained in $A_{L^p}(G)$, the map $\rho(s)\mapsto\lambda(s)$ from is continuous in the weak*-topologies from $VN_{L^p}(G)$ and $VN(G)$, respectively. Finally, Eymard showed that the map $\lambda(s)\mapsto s$ from $\lambda(G)$ to $G$ is continuous \cite[Th\'eor\'eme 3.34]{e}. Hence, we conclude the identification of $G$ with $\sigma(A_{L^p}(G))$ is a homemorphic one.
\end{proof}

The Fourier algebra admits many beautiful properties and it natural to wonder whether analogues of these continue to hold for the $L^p$-Fourier algebra. In many cases, such as with Walter's theorem, analogues do exist, but we will now see that this is not always the case.

So far we have found several classes of noncompact groups $G$ so that $A_{L^p}(G)$ is distinct for every $p\in [2,\infty)$. The following example shows in a strong way that this need not happen in general.

\begin{example}
Let $G$ be the $ax+b$ group. In 1974 Idriss Khalil demonstrated that the Fourier algebra $A(G)$ coincides with its Rajchman algebra $B_0(G):=B(G)\cap C_0(G)$ \cite{kh}. Since elements $B(G)$ are uniformly continuous, if $u\in B(G)$ is $L^p$-integrable then $u\in C_0(G)$. As $A_{L^p}(G)$ is the closed linear span of $P(G)\cap L^p(G)$ and the norm on $B(G)$ dominates the uniform norm, it follows that $A_{L^p}(G)\subset B_0(G)$. Therefore $A_{L^p}(G)=A(G)$ for every $1\leq p<\infty$.
\end{example}

\begin{remark}\label{Herz fails}
Herz's restriction theorem states that if $G$ is a locally compact group and $H$ is a closed subgroup of $G$, then $A(G)|_H=A(H)$. This example shows that the analogue of this theorem does not hold for $A_{L^p}$ when $p>2$. Indeed, $\R$ is a closed subgroup of the $ax+b$ group $G$, but $A_{L^p}(G)|_\R=A(G)|_\R=A(\R)\neq A_{L^p}(\R)$ for $p>2$ by Theorem \ref{abelian}.
\end{remark}

As previously mentioned, a locally compact group $G$ is amenabile if and only if $A(G)$ is operator amenable \cite{ru} if and only if $A(G)$ admits a bounded identity \cite{le}. These theorems fail attrociously when $A(G)$ is replaced with $A_{L^p}(G)$ for $p>2$. Indeed, our next example shows that $A_{L^p}(G)$ need not even be square dense even when $G$ is abelian.

\begin{example}
Let $G$ be a noncompact abelian group and $p>2$. Then $uv\in A_{L^{p/2}}(G)$ for all $u,v\in A_{L^p}(G)$ implies that $A_{L^p}(G)\cdot A_{L^p}(G)\subset A_{L^{p/2}}(G)$. By Theorem \ref{abelian} we know that $A_{L^{p/2}}(G)$ is strictly contained in $A_{L^p}(G)$. So $A_{L^p}(G)$ is not square dense.
\end{example}

%\begin{remark}
%As above, let $G$ be a noncompact locally compact abelian group and $p>2$. Then $A_{L^p}(G)$ is a Banach function algebra in $C_0(G)$. Since $A_{L^p}(G)$ contains $A(G)$ and $G$ is amenable, then $A_{L^p}(G)$ has a contractive pointwise approximate identity but $A_{L^p}(G)$ does not contain an approximate identity since it is not square dense. An example of a Banach function algebra with this property was first observed by Jones and Lahr REF.
%\end{remark}

As a consequence of this observation, we find that $A_{L^p}(G)$ is never an amenable Banach algebra when $G$ is noncompact and $p>2$.

\begin{prop}
Let $G$ be a locally compact group and $p>2$. If $A_{L^p}(G)\neq A(G)$, then $A_{L^p}(G)$ is not (operator) weakly amenable.
\end{prop}

\begin{proof}
Without loss of generality, we may assume that $A_{L^{p/2}}(G)\neq A_{L^p}(G)$. Indeed, if not we define
$$\widetilde p=\inf\{q\in [2,\infty) : A_{L^q}(G)=A_{L^p}(G)\}$$
and replace $p$ with $\widetilde p+\epsilon$ for some $0<\epsilon<\min\{1,p-\widetilde p\}$. Then the space $A_{L^p}(G)$ has not changed and $A_{L^p}(G)\neq A_{L^{p/2}}(G)$ since $\widetilde p> 1$ implies $p/2< (1+\widetilde p)/2<\widetilde p$. So indeed we may assume that $A_{L^p}(G)\neq A_{L^{p/2}}(G)$. Then a similar argument as in the previous example shows that $A_{L^p}(G)$ is not square dense and, therefore, is not (operator) weakly amenable.
\end{proof}

\begin{corollary}
Let $G$ be a noncompact locally compact group and $p>2$. Then $A_{L^p}(G)$ is a nonamenable Banach algebra.
\end{corollary}

\begin{proof}
By the above proposition, we may assume without loss of generality that $A_{L^p}(G)=A(G)$. Then $G$ does not contain an open abelian subgroup of finite index by Proposition \ref{open} and Theorem \ref{abelian} since such a subgroup is necessarily noncompact. In particular, this implies that $G$ is not almost abelian. Hence, $A_{L^p}(G)=A(G)$ is nonamenable \cite{fr}.
\end{proof}

Let $G_1$ and $G_2$ be locally compact groups. The Effros-Ruan tensor product formula \cite{er} implies that $A(G_1)\widehat{\otimes}A(G_2)=A(G_1\times G_2)$ where $\widehat{\otimes}$ denotes the operator projective tensor product and $u\otimes v\in A(G_1)\widehat{\otimes}A(G_2)$ is identified with $u\times v\in A(G_1\times G_2)$. The next example shows that the analogue of this formula fails for $A_{L^p}$. Before this, we observe that the algebraic tensor product $A_{L^p}(G_1)\otimes A_{L^p}(G_2)$ embeds in $A_{L^p}(G_1\times G_2)$ via the above identification.

\begin{prop}\label{kronecker}
Let $G_1$ and $G_2$ be locally compact groups and $p>2$. Then $u\times v\in A_{L^p}(G_1\times G_2)$ for all $u\in A_{L^p}(G_1)$ and $v\in A_{L^p}(G_2)$.
\end{prop}

\begin{proof}
First suppose that $u$ and $v$ are positive definite functions which are $L^p$-integrable. Then $u\times v$ is a positive definite function on $G_1\times G_2$ and
$$ \int_{G_1\times G_2} |u\times v|^p=\int_{G_1}\int_{G_2} |u(s)v(t)|^p\,ds\,dt=\|u\|_p\|v\|_p<\infty.$$
Similar arguments as used previously in the paper now show that $u\times v\in A_{L^p}(G_1\times G_2)$ for all $u\in A_{L^p}(G_1)$ and $v\in A_{L^p}(G_2)$.
\end{proof}

\begin{example}
Let $\Gamma_1$ and $\Gamma_2$ be discrete groups containing copies of nonabelian free groups and $p>2$. Then $A_{\ell^p}(\Gamma_1)\otimes A_{L^p}(\Gamma_2)$ is not norm dense in $A_{L^p}(\Gamma_1\times \Gamma_2)$. Indeed, identify copies of $\F_2$ in both $\Gamma_1$ and $\Gamma_2$ and let $\Delta$ be the diagonal subgroup of $\F_2\times \F_2\subset \Gamma_1\times \Gamma_2$. Then $u\times v|_\Delta\in A_{\ell^{p/2}}(\Delta)$ for all $u\in A_{\ell^p}(\Gamma_1)$ and $v\in A_{\ell^p}(\Gamma_2)$ by Proposition \ref{Holder} and Proposition \ref{open}. But $A_{\ell^p}(\Gamma_1\times \Gamma_2)|_\Delta=A_{\ell^p}(\Delta)$. As $A_{\ell^{p/2}}(\Gamma_1\times\Gamma_2)$ is a proper subspace of $A_{\ell^p}(\Gamma_1\times \Gamma_2)$, we conclude that $A_{\ell^p}(\Gamma_1)\otimes A_{\ell^p}(\Gamma_2)$ is not norm dense in $A_{\ell^p}(\Gamma_1)\otimes A_{\ell^p}(\Gamma)$.
%Further, we even have that $\overline{A_{\ell^p}(\Gamma_1)\odot A_{\ell^p}(\Gamma_2)}^{\sigma(B(G),C^*(G))}$ does not contain $B_{\ell^p}(\Gamma_1\times \Gamma_2)$ since $B_{\ell^p/2}(\Gamma_1\times \Gamma_2)$ is a proper subset of $B_{\ell^p}(\Gamma_1\times \Gamma_2)$.
\end{example}

The observations made in this previous example have applications to finding intermediate $C^*$-norms between the spatial and maximal tensor product norms.

\begin{theorem}\label{tensor}
Let $\Gamma_1$ and $\Gamma_2$ be discrete groups containing copies of noncommutative free groups and $p> 2$. Then $C^*_{\ell^p}(\Gamma_1\times \Gamma_2)$ gives rise to a $C^*$-norm on the algebraic tensor product $C^*_{\ell^p}(\Gamma_1)\otimes C^*_{\ell^p}(\Gamma_2)$ in the natural way. This norm is distinct from the minimal and maximal tensor product norms.
\end{theorem}

Before proving this theorem, we recall a result which we will make use of. Let $G_1$ and $G_2$ be locally compact groups with representations $\pi_1$ and $\pi_2$. We showed in \cite[Proposition 2.1]{w2} that there is a one-to-one correspondence between $C^*$-norms on the algebraic tensor product $C^*_{\pi_1}(G_1)\otimes C^*_{\pi_2}(G_2)$ and Fourier-Stieltjes spaces $B_\sigma$ of $G_1\times G_2$ such that $B_\sigma|_{G_1}=B_{\pi_1}$ $B_\sigma|_{G_2}=B_{\pi_2}$ and $B_{\sigma}\supset B_{\pi_1\times \pi_2}$. The $C^*$-norm on $C^*_{\pi_1}(G_1)\otimes C^*_{\pi_2}(G_2)$ corresponding to $B_\sigma$ is the natural one, i.e., for $f_1,\ldots,f_n\in L^1(G_1)$ and $g_1,\ldots,g_n\in L^1(G_2)$, the norm of $\sum_{k=1}^n\pi_1(f_k)\otimes \pi_2(g_k)$ is given by
$$ \bigg\|\sum_{k=1}^n \pi_1(f_k)\otimes \pi_2(g_k)\bigg\|=\bigg\|\sum_{k=1}^n \sigma\big(f_k\times g_k)\bigg\|.$$

\begin{proof}[Proof of Theorem \ref{tensor}.]
Let $\pi_1$ and $\pi_2$ be faithful $\ell^p$-representations for $C^*_{\ell^p}(\Gamma_1)$ and $C^*_{\ell^p}(\Gamma_2)$, respectively. It follows from Proposition \ref{kronecker} that $B_{\pi_1\times\pi_2}\subset B_{\ell^p}(\Gamma_1\times\Gamma_2)$ and, by Proposition \ref{open}, $B_{\ell^p}(\Gamma_1\times\Gamma_2)|_{\Gamma_1}=B_{\ell^p}(\Gamma_1)=B_{\pi_1}$ and $B_{\ell^p}(\Gamma_1\times\Gamma_2)|_{\Gamma_2}=B_{\ell^p}(\Gamma_2)=B_{\pi_2}$. So $C^*_{\ell^p}(\Gamma_1\times\Gamma_2)$ indeed induces a $C^*$-norm on $C^*_{\ell^p}(\Gamma_1)\otimes C^*_{\ell^p}(\Gamma_2)$ in the natural way. From the observations in the previous example, we have that $B_{\ell^p}(\Gamma_1)\neq B_{\pi_1\times \pi_2}$ and, hence, that the norm coming from $C^*_{\ell^p}(\Gamma_1\times\Gamma_2)$ is not the spatial tensor product norm.

Identify copies of $\F_2$ in $\Gamma_1\times \Gamma_2$ and let $\Delta$ denote the diagonal subgroup of $\F_2\times \F_2\subset \Gamma_1\times\Gamma_2$. In the proof of \cite[Theorem 3.2]{w2} a Fourier-Stieltjes space $B_\sigma$ satisfying the above conditions is constructed with the property that $B_\sigma|_{\Delta}$ contains the constant function 1. Then $B_{\ell^p}(\Gamma_1\times \Gamma_2)$ does not contain $B_\sigma$ since otherwise $B_{\ell^p}(\Delta)$ would contain the constant 1 and, hence, $B_{\ell^p}(\Delta)$ would be all of $B(\Delta)$. This would be a contradiction since $\Delta\cong \F_2$ is nonamenable.
\end{proof}

In a previous example we observed that characterizations of amenability in terms of the Fourier algebra can fail when $A(G)$ is replaced with $A_{L^p}(G)$. We finish this section by identifying some characterizations of amenability which do translate over.

Let $\mc A$ be a Banach algebra. A linear operator $T\fn \mc A\to \mc A$ is said to be a multiplier of $\mc A$ if $T(ab)=aT(b)=T(a)b$ for all $a,b\in \mc A$. In the context of Fourier algebras $A(G)$, every multiplier every multiplier is bounded and is realizable as multiplication by some function on $G$ \cite{lo}. Viktor Losert characterized the amenability of a locally compact group $G$ in terms of multipliers by showing that $G$ is amenable if and only if $M(A(G))$, the set of multipliers of $A(G)$, is exactly $B(G)$ if and only if the norm on $A(G)$ is equivalent to the norm it attains as a multiplier on itself \cite[Theorem 1]{lo}. We show that the analogue of this theorem holds for $L^p$-Fourier algebras.

\begin{theorem}
The following are equivalent for a locally compact group $G$ and $1\leq p<\infty$.
\begin{itemize}
\item[(i)] $G$ is amenable.
\item[(ii)] $M(A_{L^p}(G))=B(G)$.
\item[(iii)] $\|\cdot\|_{B(G)}$ is equivalent to $\|\cdot\|_{M(A_{L^p}(G))}$ on $B(G)$.
\item[(iv)] $\|\cdot\|_{B(G)}$ is equivalent to $\|\cdot\|_{M(A_{L^p}(G))}$ on $A(G)$.
\item[(v)] $\|\cdot\|_{B(G)}$ is equivalent to $\|\cdot\|_{M(A(G))}$ on $A(G)$.
\end{itemize}
\end{theorem}

\begin{proof}
{(i)} $\Rightarrow$ {(ii)}: It is an application of the closed graph theorem that every element of $M(A_{L^p}(G))$ is bounded and given by a multiplication operator. Suppose that $v\in C_b(G)$ is a multiplier of $A_{L^p}(G)$. Since $G$ is amenable, $A(G)$ admits a bounded pointwise approximate identity $\{u_\alpha\}$. Then $\{u_\alpha v\}$ is a bounded sequence converging pointwise to $v$ and, hence, $v\in B(G)$ \cite[Corollaire 2.25]{e}.

{(ii)} $\Rightarrow$ {(iii)}: Standard application of the open mapping theorem.

{(iii)} $\Rightarrow$ {(iv)}: Clear.

{(iv)} $\Rightarrow$ {(v)}: Suppose that there exists $c>0$ so that
$$ \sup\{\|u v\|_{B(G)} : v\in A_{L^p}(G), \|v\|_{B(G)}\leq 1\}> c\|u\|_{B(G)} $$
for every $u\in A(G)$ and choose $n$ sufficiently large so that $p/n<2$. Fix $u\in A_{L^p}(G)$ and choose a unit vector $v_1$ in $A_{L^p}(G)$ so that $\|uv_1\|_{B(G)}>c\|u\|_{B(G)}$. Next choose $v_2\in A_{L^p}(G)$ so that
$$\|(uv_1)v_2\|_{B(G)}>c \|uv_1\|_{B(G)}>c^2 \|u\|_{B(G)}.$$
Repeat this process until we arrive at $n$ unit vectors $v_1,\ldots,v_n\in A_{L^p}(G)$ and define $v=v_1\cdots v_n$. Then $v\in A(G)$ has norm at most 1 and $\|uv\|_{B(G)}>c^n\|u\|_{B(G)}$. Hence, $\|\cdot\|_{B(G)}$ is equivalent to $\|\cdot\|_{M(A(G))}$ on $A(G)$.

{(v)} $\Rightarrow$ {(i)}: As mentioned above, this was shown by Losert.
\end{proof}

We now prove a characterization of amenability in terms of the $L^p$-Fourier-Stieltjes algebra. Recall that

In \cite{rs} Volker Runde and Nico Spronk introduced the notion of operator Connes amenability and showed that a locally compact group $G$ is amenable if and only if the reduced Fourier-Stieltjes algebra $B_\lambda$ is operator Connes amenable. Since $B_{L^p}(G)$ is the dual space of $C^*_{L^p}(G)$, it also has a natural operator space structure. We finish this section by showing that this characterization holds for $L^p$-Fourier-Stieltjes algebras.

\begin{theorem}
Let $G$ be a locally compact group and $p\in [1,\infty)$. Then $G$ is amenable if and only if $B_{L^p}(G)$ is operator Connes amenable.
\end{theorem}

\begin{proof}
First suppose that $G$ is amenable. Then $B_{L^p}(G)=B_\lambda(G)=B(G)$ is operator Connes amenable \cite[Theorem 4.4]{rs}.

Next suppose that $B_{L^p}(G)$ is operator Connes amenable. Then, as in the proof of \cite[Theorem 4.4]{rs}, $B_{L^p}(G)$ has an identity. So $B_{L^p}(G)=B(G)$ and, hence, $G$ is amenable by Proposition \ref{basic}.
\end{proof}

\section{Fourier-Stieltjes ideals of $SL(2,\R)$} \label{SL(2,R)}

In this section, we study the $L^p$-Fourier-Stieltjes algebras for $SL(2,\R)$ and characterize the Fourier-Stieltjes ideals of $SL(2,\R)$. The representation theory of $SL(2,\R)$ is very well understood, and this knowledge is used intimately throughout this section.

The irreducible representations of $SL(2,\R)$ fall into the following five categories:
\begin{center}
\begin{tabular}{r c l}
Trivial representation & : & $\tau$,\\
Discrete series & : & $\{T_n : n\in\Z, |n|\geq 2\}$,\\
Mock discrete series & : & $T_{-1}, T_1$,\\
Principal series & : & $\{\pi_{it,\epsilon} : t\in\R, \epsilon=\pm 1\}$,\\
Complementary series & : & $\{\pi_r : -1<r<0\}$.
\end{tabular}
\end{center}
There is no standard for the notation and parameterizations of these representations, so, for convenience, we will follow that used by Joe Repka in \cite{r} -- a paper which will refer to again. The Fell topology on these representations is also completely understood. Rather than detailing this topology we refer the reader to Folland's book \cite[Figure 7.3]{f} for a nice description.

Ray Kunze and Elias Stein studied the integrability properties of the coefficients of irreducible representations $SL(2,\R)$ and demonstrated the remarkable fact that for every nontrivial irreducible representation $\pi$ of $SL(2,\R)$, there exists a $p\in [2,\infty)$ so that $\pi_{x,x}\in L^p$ for every $x\in\Hi_\pi$. In fact, for an irreducible representation $\pi$ of $SL(2,\R)$ they showed: \cite[Theorem 10]{ks}
 \begin{itemize}
 \item $\pi$ is an element of the discrete series if and only if every coefficient function of $\pi$ is $L^2$-integrable,
 \item $\pi$ is an element of the mock discrete series or the continuous principal series if and only if every coefficient function of $\pi$ is $L^{2+\epsilon}$-integrable for every $\epsilon>0$, but not every coefficient function is $L^2$-integrable,
 \item $\pi$ is an element of the complementary series with parameter $r\in (-1,0)$ if and only if every coefficient function of $r$ is $L^{2/(1+r)+\epsilon}$-integrable for every $\epsilon>0$, but not every coefficient function is $L^{2/(1+r)}$-integrable.
 \end{itemize}
A fortiori, every nontrivial irreducible representation of $SL(2,\R)$ is an $L^p$-representation for some $p\in [2,\infty)$. We use this and a result of Repka to show that the spaces $B_{L^p}(SL(2,\R))$ are distinct for every $p\in [2,\infty)$.

\begin{lemma}\label{irreps weakly contained in Lp for SL(2,R)}
Let $G$ be the group $SL(2,\R)$. Then
\begin{itemize}
\item[(i)] The discrete series, mock discrete series, and principal series are weakly contained in the $L^p$-representations for every $p\in [2,\infty)$.
\item[(ii)] The complementary series representations $\pi_r$ is weakly contained in the $L^p$-representations (for $p\in [2,\infty)$ if and only if $r\in [2/p-1,0)$
\end{itemize}
\end{lemma}

\begin{proof}
Let $\pi$ be a representation of $SL(2,\R)$. Then \cite[Theorem 9.1]{r} immediately implies that if $\pi$ is an $L^p$-representation for some $p>2$, then the direct integral decomposition of $\pi$ does not include the representations $\pi_r$ for $-1<r<2/p-1$ (apart from on a null set). Hence, $\pi$ does not weakly contain $\pi_r$ for any $-1<r<2/p-1$ since the set $\{\pi_r : -1<r<2/p-1\}$ is open in the Fell topology.

Note that by the results of Kunze and Stein mentioned above, $\pi_r$ is an $L^p$-representation for every $2/p-1<r<0$. Hence, the $L^p$-representations weakly contain $\pi_r$ if and only if $2/p-1\leq r<0$.

To complete our proof we must note that the mock discrete series and principal series are weakly contained in the left regular representation. But this is given by the Cowling-Haagerup-Howe theorem \cite{cch} since they are each $L^{2+\epsilon}$-representations for every $\epsilon>0$.
\end{proof}

\begin{corollary}
Let $G=SL(2,\R)$. Then the Fourier-Stieltjes spaces $B_{L^p}(G)$ are distinct for every $p\in [2,\infty)$. Equivalently, the $C^*$-algebras $C^*_{L^p}(G)$ are distinct for every $p\in [2,\infty)$.
\end{corollary}

We now proceed to prove the main result of this section: a characterization of the Fourier-Stieltjes ideals of $SL(2,\R)$.

\begin{theorem}
Let $I$ be a nontrivial Fourier-Stieltjes ideal of $SL(2,\R)$. Then $I=B(G)$ or $I=B_{L^p}(G)$ for some $p\in [2,\infty)$.
\end{theorem}

\begin{proof}
Write $I=B_\pi$ for some representation $\pi$ of $G$. Then, since $\pi\otimes \lambda$ is unitarily equivalent to an amplification of $\lambda$ by Fell's absorption principle, it is an easy exercise to see that $B_\pi\supset B_\lambda=B_{L^2}(G)$.

Consider the case when the trivial representation $\tau$ is weakly contained in $\pi$. Then $I$ contains the unit and, hence, is all of $B(G)$.

Next consider the case when $\pi$ does not contain the complementary representation $\pi_r$ for any $r\in (-1,0)$. Then, by Lemma \ref{irreps weakly contained in Lp for SL(2,R)}, $B_\pi$ is a subset of $B_{L^2}(G)$. Since we already know the reverse inclusion, we conclude that $I=B_{L^2}(G)$.

Finally, we consider the case when $\pi$ weakly contains some element of the complementary series. Let
$$r=\inf\{r'\in (-1,0) : \pi_{r'}\tn{ is weakly contained in }\pi\}.$$
Then $r>-1$ since $\pi_r$ converges to the trivial representation $\tau$ in the Fell topology as $r\to -1$. Also notice that $\pi$ weakly contains $\pi_r$ since $\pi_{r'}\to \pi_r$ in the Fell topology as $r'\to r$. In \cite{p} Puk\'anszky showed that if $r_1,r_2\in (-1,0)$ with $r_1+r_2<-1$, then $\pi_{r_1+r_2+1}$ is a subrepresentation of $\pi_{r_1}\otimes\pi_{r_2}$ (see also \cite[Theorem 5.9]{r}). Since $r+r'<-1$ for $-1<r'<-r-1$ and $\pi$ weakly contains $\pi_{r}\otimes \pi_{r'}$ for every $-1<r<0$, it follows that $\pi$ weakly contains $\pi_{r'}$ for each $r\leq r'<r$. Therefore, by Lemma \ref{irreps weakly contained in Lp for SL(2,R)}, we conclude that $I=B_{L^p}(G)$ where $p=2/(1+r)$.
\end{proof}

It is natural to wonder which other groups are the Fourier-Stieltjes ideals characterizable as above. Unfortunately this characterization does not hold for arbitrary locally compact groups $G$ per \cite{w}.

\begin{example}
Consider the free group $\F_\infty$ on countably many generators $a_1,a_2,\ldots$ and let $\F_d$ denote the subgroup generated by $a_1,\ldots,a_d$ for some $2\leq d<\infty$. For each $p\in [1,\infty)$, define
$$ D_p=\{f\in \ell^\infty(\F_\infty) : f|_{s\F_d t}\in \ell^p(s\F_d t) \tn{ for all }s,t\in \F_\infty\}.$$
Then $D_p$ is an ideal of $\ell^\infty(\F_\infty)$ which implies that $B_{D_p}$ is an ideal of $B(G)$. Moreover, in \cite[Example 2.5]{w} it was shown that $C^*_{D_p}(\F_\infty)\neq C^*_{\ell^q}(\F_\infty)$ for any $1\leq q<\infty$ and that $C^*_{D_p}(\F_\infty)$ is distinct for each $p\in [2,\infty)$. Hence, $\F_\infty$ has a continuum of Fourier-Stieltjes ideals which are not of the form $B_{\ell^p}(\F_\infty)$ for some $p\in [2,\infty)$.
\end{example}

\section*{Acknowledgements}
The author would like to thank his supervisor Brian Forrest for suggesting this problem. The author also wishes to thank both his advisors, Brian Forrest and Nico Spronk, for their valuable input. This research was conducted at the Fields Institute during the thematic program on Abstract Harmonic Analaysis, Banach and Operator Algebras while the author was supported by an NSERC Postgraduate Scholarship.

\end{document}